\documentclass{amsart}
\usepackage{pdflscape, comment}
\usepackage{amsmath, amssymb, amsthm}\usepackage{tikz}

\newtheorem{theorem}{Theorem}[section]
\newtheorem{lemma}{Lemma}[section]

\newcommand{\Alt}{\mathrm{Alt}}
\newcommand{\Sym}{\mathrm{Sym}}
\newcommand{\PSL}{\mathrm{PSL}}
\newcommand{\PGL}{\mathrm{PGL}}
\newcommand{\PG}{\mathrm{PG}}
\newcommand{\PSU}{\mathrm{PSU}}
\newcommand{\PSO}{\mathrm{PSO}}
\newcommand{\POmega}{\mathrm{P\Omega}}
\newcommand{\PGaL}{\mathrm{P\Gamma L}}
\newcommand{\FF}{\mathbb{F}}

\newcommand{\set}[1]{\left\{ #1 \right\}}

\title{On chiral polytopes having a group $\PSL(3,q)$ as automorphism group}
\author[Dimitri Leemans]{Dimitri Leemans}
\address{Dimitri Leemans, D\'epartement de Math\'ematique, Universit\'e libre de Bruxelles, C.P.216, Boulevard du Triomphe, 1050 Brussels, Belgium}
\email{dleemans@ulb.ac.be}
\author[Adrien Vandenschrick]{Adrien Vandenschrick}
\address{Adrien Vandenschrick, FNRS-FRS Research Fellow, D\'epartement de Math\'ematique, Universit\'e libre de Bruxelles, C.P.216, Boulevard du Triomphe, 1050 Brussels, Belgium}
\email{adrien.vandenschrick@ulb.ac.be}

\date{\today}

\begin{document}

\maketitle

\begin{abstract}
    For each prime power $q\geq 5$, we construct a rank four chiral polytope that has a group $\PSL(3,q)$ as automorphism group and Schl\"afli type $\{q-1,\frac{2(q-1)}{(3,q-1)},q-1\}$. We also construct rank five polytopes for some values of $q$ and we show that there is no chiral polytope of rank at least six having a group $\PSL(3,q)$ or $\PSU(3,q)$ as automorphism group.
\end{abstract}
\section{Introduction}

The finite simple groups have for a long time been a subject of investigation. Despite the Classification Theorem, one has to admit that a lot has yet to be understood regarding them. For instance, there is still the need to find a unified geometric interpretation of these groups. The Theory of Buildings, due to Jacques Tits (see for instance~\cite{Tits74}) permits to associate a geometric object to most of the nonabelian simple groups. Alternating groups and sporadic groups are however not covered by that theory.

It is known which finite simple groups are automorphism groups of abstract regular polytopes. Indeed, Nuzhin~\cite{n1,n2,n3,n4} and Mazurov~\cite{Mazurov03} showed that every non-abelian finite simple group can be generated by three involutions, two of which commute, with the following exceptions:
\[
\begin{array}{l}
\PSL(3,q),\, \PSU(3,q),\, \PSL(4,2^n), \, \PSU(4,2^{n}),\\
\Alt(6), \, \Alt(7), 
%{\color{red}\,PSU_4(3), \,PSU_5(2), }
\, M_{11},\, M_{22}, \,M_{23}, \,McL.
\end{array}
\]
%{\color{red}
The groups $\mathrm{PSU(4,3)}$ and $\mathrm{PSU(5,2)}$, although mentioned by Nuzhin as being generated by three involutions, two of which commute, have been discovered not to have such generating sets by Martin Macaj and Gareth Jones (personal communication).
Recently, Vandenschrick~\cite{Adrien} showed that all the exceptions in rank three are also not automorphism groups of abstract regular polytopes of higher ranks.

In the chiral case, Leemans and Liebeck~\cite{LL2017} determined that all finite simple groups are automorphism groups of abstract chiral polyhedra with the following exceptions:
\[
\begin{array}{l}
 \PSL(2,q), \, \PSL(3,q), \, \PSU(3,q), \, \Alt(7).
 \end{array}
\]
The group $\Alt(7)$ is known not to be the automorphism group of any chiral polytope (see for instance~\cite{HHL}).
It was already known by Macbeath~\cite{Macbeath} that the groups $\PSL(2,q)$ are not automorphism groups of chiral polyhedra. Leemans, Moerenhout and O'Reilly-Regueiro~\cite{LM2017} showed that there are no chiral polytopes of rank five or higher having a group $\PSL(2,q)$ as automorphism group. They also showed that for most values of $q$, the group $\PSL(2,q)$ is the automorphism group of a rank four chiral polytope.

The groups $\PSL(3,q)$ and $\PSU(3,q)$ have been shown not to be automorphism groups of chiral polyhedra by Breda and Catalano~\cite{BC2019}.

%We consider here the relatively easy group $\PSL(3,q)$. We will look at this group through the theory of abstract polytopes hopefully providing some interesting results in those two fields.

%The projective group $\PSL(3,q)$ is a group of automorphism of the projective plane, but we will try to describe it as the group of automorphisms of a polytope. We will consider two kinds of polytopes : regular and chiral. In those two situations, involutions play a central role in their automorphism group as it does in the classification of the finite simple groups.

We show in this paper that for $q\geq 5$, the group $\PSL(3,q)$ is the automorphism group of at least one chiral polytope of rank four. 

\begin{theorem}\label{rank4}
Let $\FF_q$ ($q \geq 5$) be the finite field of order $q$.
    The following elements of $\PSL(3,q)$, with $x$ a primitive element of $\FF_q$, generate a chiral polytope of rank four.
    \[
    \sigma_1 = \begin{pmatrix}
    x & 1 & 0 \\
    0 & 1 + x^{-1} & -x^{-1} \\
    0 & 1 & 0
    \end{pmatrix},\ 
    \sigma_2 = \begin{pmatrix}
    -x^{-1} & 0 & 0 \\
    0 & 0 & 1 \\
    0 & x & 0
    \end{pmatrix}
    \text{and }
    \sigma_3 = \begin{pmatrix}
    x & 0 & 0 \\
    x - x^{-1} & 1 & 0 \\
    x - x^{-1} & 0 & x^{-1}
    \end{pmatrix}
    \]
    Moreover, $\left\langle \sigma_1,\ \sigma_2 \right\rangle$ is the stabiliser of two lines of the projective plane $\PG(2,q)$ while $\left\langle \sigma_2,\ \sigma_3 \right\rangle$ is the stabiliser of two points. 
    Finally, this polytope is self-dual of Schl\"afli type $\{q-1,\frac{2(q-1)}{(3,q-1)},q-1\}$.

\end{theorem}

We also construct rank five chiral polytopes for some groups $\PSL(3,q)$.

\begin{theorem}\label{rank5}
   Let $q = p$ be a prime number congruent to $1$ modulo $6$. 
   Let $\FF_q$ be the finite field of order $q$.
   Let $\omega$ be a primitive third root of unity of $\FF_q$. 
   Then the following elements of $\PSL(3,q)$ generate chiral polytopes of rank five.
    \[
    \sigma_1 = \begin{pmatrix}
    1 & 0 & 0 \\
    0 & \omega^i & 0 \\
    0 & 0 & \omega^{2i}
    \end{pmatrix},\ 
    \sigma_2 = \begin{pmatrix}
    -1 & 0 & -k \\
    0 & -\omega^{2i} & -k\omega^{2i} \\
    0 & 0 & \omega^i
    \end{pmatrix},\ 
    \]
    
    \[
    \sigma_3 = \begin{pmatrix}
    1 & k & 0 \\
    0 & k & 1 \\
    0 & -1 & 0
    \end{pmatrix}
    \text{ and }
    \sigma_4 = \begin{pmatrix}
    0 & 1 & 0 \\
    0 & 0 & 1 \\
    1 & 0 & 0
    \end{pmatrix}
    \]
    with $k = \pm 1$ and $i=1$ or $2$. The Schl\"afli types are $\{3,6,6,3\}$ for $k=1$ and $\{3,6,3,3\}$ for $k = -1$.
\end{theorem}

Finally, we show that there are no chiral polytopes of rank greater than five having a group $\PSL(3,q)$ as automorphism group. Our proof actually works for the groups $\PSU(3,q)$.

\begin{theorem}\label{rank6}
Let $\FF_q$ ($q \geq 2$) be the finite field of order $q$.
The groups $\PSL(3,q)$ and $\PSU(3,q)$ are not automorphism groups of chiral polytopes of rank greater than five.
\end{theorem}

The paper is organised as follows.
In Section~\ref{prelim}, we give all the necessary definitions and notation to understand this paper.
In Section~\ref{srank4}, we prove Theorem~\ref{rank4}.
In Section~\ref{srank5}, we prove Theorem~\ref{rank5}.
In Section~\ref{regular}, we classify the rank four directly regular polytopes that have a subgroup of $\PSL(3,q)$ as automorphism group. These are needed in order to prove Theorem~\ref{rank6}.
In Section~\ref{srank6}, we prove Theorem~\ref{rank6}.
Finally, we conclude the paper in Section~\ref{conclusion} where we give some open problems.
\section{Preliminaries}\label{prelim}

Groups of automorphisms of regular and chiral polytopes have been characterised in~\cite[Section 2E]{McMullenSchulte} and~\cite{SchulteWeiss} respectively.

A group $G$ is the group of automorphism of a regular polytope of rank $r$ if and only if there exists a tuple $S := (\rho_0,\cdots,\rho_{r-1})$ of involutions of $G$ which generate $G$ and such that
\begin{itemize}
    \item for every $i,j\in\{0, \ldots, r-1\}$, $\rho_i$ and $\rho_j$ commute if $|i-j|>1$ (string property) and
    \item for every $I$, $J\subseteq \{0,\ldots ,r-1\}$, $\left\langle \rho_i | i \in I \right\rangle \cap \left\langle \rho_i | i \in J \right\rangle = \left\langle \rho_i | i \in I \cap J \right\rangle $ (intersection property $IP$)
\end{itemize}
The pair $(G,S)$ is called a {\em string $C$-group representation} of $G$. 
The {\em Schl\"afli type} of the polytope or of the string C-group representation is the tuple formed by the orders of $\rho_i\rho_{i+1}$ for $i = 0, \cdots, r-2$.
If $(G,S)$ has a Schl\"afli type with $2$'s in it, it is called {\em degenerate}. Otherwise it is called {\em non-degererate}.
Moreover, $(G,S)$ is {\em directly regular} if the subgroup $\left\langle \rho_0\rho_1, \rho_1\rho_2, \cdots, \rho_{r-2}\rho_{r-1} \right\rangle$ has index $2$ in $G$. This subgroup is then called the {\em rotation subgroup}. The {\em dual string $C$-group representation} is obtained by taking the same group $G$ with the tuple $(\rho_{r-1},\rho_{r-2},\cdots,\rho_0)$.

A group $G$ is the group of automorphisms of a chiral polytope of rank $r$ if and only if there exists a tuple $(\sigma_1,\cdots,\sigma_{r-1})$ such that the $\sigma_i$'s generate $G$ and such that
\begin{itemize}
    \item for every $i,j\in\{1,\ldots, r-1\}$, if $i<j$, the products $\tau_{i,j} = \sigma_i\sigma_{i+1}\cdots\sigma_j$ are involutions (string property),
    \item $\left\langle \tau_{i,j} | i,j \notin I \right\rangle \cap \left\langle \tau_{i,j} | i,j \notin J \right\rangle = \left\langle \tau_{i,j} | i,j \notin I \cup J \right\rangle$ (intersection property $IP^+$) and
    \item there does not exist an automorphism sending $\sigma_1$ to its inverse and fixing $\tau_{1,i}$ for $i = 2, \ldots, r-1$.
\end{itemize}
The group $G$ together with this tuple is then said to form a {\em $C^+$-group} as in~\cite{FLW}.
The {\em Schl\"afli type} of the polytope is the tuple formed by the orders of the $\sigma_i$'s for $i = 1, \cdots, r-1$. 
The dual $C^+$-group is obtained by taking the same group $G$  with the tuple $(\sigma_{r-1}^{-1},\sigma_{r-2}^{-1},\cdots,\sigma_1^{-1})$.

If there exist a tuple verifying all these properties except the last one, then the group generated by $G$ and the automorphism of the last property is the group of automorphism of a directly regular polytope. Moreover, $G$ is the rotation subgroup of that polytope.

In both cases the verification of the intersection property can be shortened if one assumes the other properties. The verification is recursive in nature. We have the following results.
\begin{lemma}\cite[Proposition 2E16]{McMullenSchulte}
   Let $G$ be a group and let $S:=(\rho_1,\cdots,\rho_r)$ be a group of involutions that generate $G$ and satisfy the string property. Then the pair $(G,S)$ satisfies the intersection property $IP$ if and only if
   \[ (\rho_1,\cdots,\rho_{r-1}) \text{ and } (\rho_2,\cdots,\rho_{r})\]
   satisfy the intersection property and
   \[ \left\langle \rho_1,\cdots,\rho_{r-1}\right\rangle
   \cap  \left\langle \rho_2,\cdots,\rho_{r}\right\rangle 
   =  \left\langle \rho_2,\cdots,\rho_{r-1}\right\rangle. \]
\end{lemma}

\begin{lemma}\cite[Lemma 10]{SchulteWeiss}
   Let $G$ be a group and let $S:= (\sigma_1,\cdots,\sigma_{r-1})$ be a set of elements of $G$ that generate $G$ and satisfy the string property.
   Then $(G,S)$ verifies the intersection property $IP^+$ if and only if
   \[ (\sigma_1,\cdots,\sigma_{r-2})\]
   satisfies the intersection property $IP^+$ and
   \[ \left\langle \sigma_1,\cdots,\sigma_{r-2}\right\rangle
   \cap  \left\langle \sigma_i,\cdots,\sigma_{r-1}\right\rangle 
   =  \left\langle \sigma_i,\cdots,\sigma_{r-2}\right\rangle \text{ for } i = 1,\cdots, r-1.\]
\end{lemma}

We can construct the automorphism group of a chiral polytope from the automorphism groups of directly regular polytopes by relying on the following lemma which is just a rewriting, in terms of $C^+$-groups of the fact that chiral polytopes of rank $n$ have $(n-2)$-faces that are directly regular polytopes, as it was shown in~\cite{SchulteWeiss}.
\begin{lemma}\label{lem:2.3}
   Let $(G,(\sigma_1, \ldots, \sigma_{r-1}))$ be a $C^+$-group.
   Then the subgroup (where $\tau_{i,j}$ is defined as above)
   \[ G^+ := \left\langle \tau_{1,2}, \tau_{1,3}, \cdots, \tau_{1,r-1} \right\rangle \]
   together with the tuple consisting of its generators is a string $C$-group representation of $G^+$ that is directly regular, with rotation subgroup
   \[ \left\langle \sigma_{3}, \sigma_4 \cdots, \sigma_{r-1} \right\rangle. \]
   From a $C^+$-group of type $[k_1,\cdots,k_{r-1}]$, we get a string $C$-group of type $[k_3,\cdots,k_{r-1}]$. The dual construction yields a string $C$-group of type $[k_1,\cdots,k_{r-3}]$.
\end{lemma}
This property shows in particular that to construct a $C^+$-group representation for $G$ of rank $r$, it may be useful to know string $C$-group representations of rank $r-2$ for subgroups of $G$. It is why we study in Section~\ref{regular} the string C-group representations of subgroups of $\PSL(3,q)$.

We see the group $\PSL(3,q)$ as a permutation group acting on a projective plane $\PG(2,q)$ with the natural action.
We denote by $\epsilon$ the identity element of $\PSL(3,q)$.

\section{Chiral polytopes of rank four for $\PSL(3,q)$}\label{srank4}
We give a constructive proof of Theorem~\ref{rank4}, that is, we give sets of triples of elements of $\PSL(3,q)$ that satisfy the conditions stated in the previous section.
As stated in Theorem~\ref{rank4}, in this section, $q\geq 5$.

\begin{proof}[Proof of Theorem~\ref{rank4}]
Suppose there exist generators $\sigma_1$, $\sigma_2$ and $\sigma_3$ of $\PSL(3,q)$ such that $(\PSL(3,q),\allowbreak(\sigma_1, \sigma_2, \sigma_3))$ gives a chiral polytope of rank four.
Suppose moreover that
$\left\langle \sigma_1,\ \sigma_2 \right\rangle$ and $\left\langle \sigma_2,\ \sigma_3 \right\rangle$ are stabilisers of respectively two lines ($l$ and $l'$) and two points ($p$ and $p'$). Suppose that $\sigma_1$ stabilises both lines and $\sigma_3$ stabilises both points.

We will prove that, up to isomorphism, there is a unique family of choices for the triple $\{\sigma_1, \sigma_2, \sigma_3\}$ indexed by a parameter $x$.

If $(\PSL(3,q),(\sigma_1, \sigma_2, \sigma_3))$ gives a chiral polytope, the triple 
$(\sigma_1, \sigma_2, \sigma_3)$ satisfies the intersection property.
Hence the intersection of $\left\langle \sigma_1,\ \sigma_2 \right\rangle$ and $\left\langle \sigma_2,\ \sigma_3 \right\rangle$ has to be a cyclic group.
This gives some restrictions on the lines and points we can choose.
We can assume that $p$ and $p'$ do not belong to any of the two lines. Hence, we may choose a system of coordinates of $\PG(2,q)$ such that
\begin{itemize}
    \item $(1:0:0)$ is the intersection of the two lines $l$ and $l'$,  
    \item $(0:1:0)$ is the point $p$,  
    \item $(0:0:1)$ is the point $p'$,  
    \item $(0:1:1)$ is the intersection of the two lines $l$ and $pp'$ and
    \item $(0:1:x)$ is the intersection of the two lines $l'$ and $pp'$ (for some $x \not= 0,1$). 
\end{itemize}
Let us remark that $(\sigma_1\sigma_2\sigma_3)^2 = \epsilon$ and $(\sigma_1\sigma_2)^2 = \epsilon$ imply that $(\sigma_1\sigma_2)\sigma_3(\sigma_1\sigma_2)^{-1} = \sigma_3^{-1}$. Therefore $\sigma_1\sigma_2$ induces an involution on the orbits of $\sigma_3$ (which are the same as the orbits of $\sigma_3^{-1}$). In particular, it induces an involution on the fixed points and fixed lines of $\sigma_3$.

%If $x = 0$, there exists a primitive element $a\in \FF_q$ such that
%\[ \tau_2 = \begin{pmatrix}
%    a^{-2} & 0 & 0 \\
%    0 & a & 0 \\
%    0 & 0 & a
%    \end{pmatrix} \]
%Since, $\tau_2$ stabilises the lines passing through $(1:0:0)$, $\tau_1$ has to exchange $l$ and $l'$. Hence $\tau_1$ induces an involution on the lines passing through $(1:0:0)$. Therefore, $\left\langle \tau_1,\ \tau_2 \right\rangle$ is not transitive on the lines passing through $(1:0:0)$ which are different from $l$ and $l'$. Hence, we can assume $x \not= 0$.

We have $x \not= -1$ and, for some elements $a,b$ of the field such that $b = ax$ and $ab$ is primitive, we have
\[ \sigma_2 = \begin{pmatrix}
    -a^{-1}b^{-1} & 0 & 0 \\
    0 & 0 & a \\
    0 & b & 0
    \end{pmatrix}.\]
%Let us assume that $\sigma_1$ fixes $l$ and $l'$ while $\sigma_3$ fixes $p$ and $p'$.
Since the products $\sigma_1\sigma_2$ and $\sigma_2\sigma_3$ are involutions, we get that
\[ 
    \sigma_1^C = \begin{pmatrix}
    ab & uy & u \\
    0 & a^{-1}y & 0 \\
    0 & 0 & b^{-1}y^{-1}
    \end{pmatrix}
    \text{ and }
    \sigma_3 = \begin{pmatrix}
    ab & 0 & 0 \\
    vz & a^{-1}z & 0 \\
    v & 0 & b^{-1}z^{-1}
    \end{pmatrix}
    \text{ with }
    C = \begin{pmatrix}
    1 & 0 & 0 \\
    0 & 1 & 1 \\
    0 & 1 & x
    \end{pmatrix}.
\]

Since $\sigma_3$ is conjugate to its inverse, we have a restriction on its eigenvalues. By taking a suitable representative, its eigenvalues have the form $1,\lambda,\lambda^{-1}$. Note that $\lambda$ is a primitive element for otherwise $\left\langle \sigma_2,\ \sigma_3 \right\rangle$ would not act transitively on the line $pp'$ whose points $p$ and $p'$ have been removed. As $q\geq 5$, we have that $\lambda \not= \pm 1$ and $\sigma_3$ has three different eigenvalues. 

Since $\sigma_1\sigma_2$ induces an involution on the three fixed points of $\sigma_3$, it can stabilise $p$ or stabilise $p'$ or exchange $p$ and $p'$. Therefore, $\sigma_1$ can send $p$ to $p'$ or send $p'$ to $p$ or stabilise both $p$ and $p'$. In the first case, $y=x^{-1}$. In the second case, $y=1$. In the third case $\left\langle \sigma_1,\ \sigma_2,\ \sigma_3\right\rangle$ would stabilise the set $\left\{p,p' \right\}$ and hence  $\left\langle \sigma_1,\ \sigma_2,\ \sigma_3\right\rangle\neq \PSL(3,q)$, a contradiction. Using a similar proof, one can see that $\sigma_3$ either sends $l$ to $l'$ (case $z=x^{-1}$) or sends $l'$ to $l$ (case $z=1$).

By eventually exchanging the roles of $p$ and $p'$ and the roles of $l$ and $l'$, we can make sure that $\sigma_1$ sends $p'$ to $p$ and $\sigma_3$ sends $l'$ to $l$. We then have
\[ \sigma_1 =
\begin{pmatrix}
    ab & u & 0 \\
    0 & a^{-1}+b^{-1} & - b^{-1} \\
    0 & a^{-1} & 0 \\
\end{pmatrix}
\text{ and }\sigma_3 = \begin{pmatrix}
    ab & 0 & 0 \\
    v & a^{-1} & 0 \\
    v & 0 & b^{-1}
\end{pmatrix}.\]

The condition $(\sigma_1\sigma_2\sigma_3)^2 = \epsilon$ implies that $a^3 = 1$. Hence, by choosing a good representative for the projective transformation $\sigma_2$, one can assume $a=1$ and $b=x$. Note that $ab=x$ has to be primitive. Finally, this condition also yields $uv=x-x^{-1}$. By choosing a suitable basis, we can assume $u=1$ which gives the three generators in the statement of the theorem.

Now, we have to verify that these three generators satisfy the intersection property.
Hence, we have to verify the following three equalities.
\[ \left\langle \sigma_1 \right\rangle \cap \left\langle \sigma_2 \right\rangle = \set{\epsilon},\quad
\left\langle \sigma_2 \right\rangle \cap \left\langle \sigma_3 \right\rangle = \set{\epsilon}\ \text{and}\ 
\left\langle \sigma_1,\ \sigma_2 \right\rangle \cap \left\langle \sigma_2,\ \sigma_3 \right\rangle =\left\langle \sigma_2 \right\rangle.
\]
Let us remark that the powers of $\sigma_2$ either acts as an involution on the line $pp'$ (which exchanges $p$ with $p'$ and $(0:1:1)$ with $(0:1:x)$) or fixes every point of the line $pp'$. This is not the case for any power of $\sigma_1$ or $\sigma_3$ which is not $\epsilon$. This proves the first two equalities. For the last equality, one can see that with our choice of $\sigma_2$ it generates all transformations which stabilise $\set{p,p'}$ and $\set{l,l'}$. Moreover, $\sigma_1$ and $\sigma_2$ are transformations which stabilise $\set{l,l'}$. Similarly, $\sigma_2$ and $\sigma_3$ are transformations which stabilise $\set{p,p'}$.

The duality is given by
\[ \phi : \PSL(3,q) \to \PSL(3,q) : \sigma \to D^{-1} \sigma^{-t} D \text{ with }
D = \begin{pmatrix}
    \left( x-x^{-1} \right)^{-1} & 0 & 0 \\
    0 & 1 & 1 \\
    0 & 1 & x
    \end{pmatrix}\]

We now prove that $\left\langle \sigma_2,\ \sigma_3 \right\rangle$ is the whole group $H$ of projective transformations which stabilise $\set{p,p'}$. Let us first examine $H$. It contains a subgroup $H_1$ of index $2$ which consists of transformations that fix both $p$ and $p'$. This subgroup $H_1$ is the semi-direct product of a normal subgroup $K$ of order $q^2$ and a subgroup $L$ of order $\frac{(q-1)^2}{(3,q-1)}$. The subgroup $K$ consists of nilpotent transformations while the subgroup $L$ consists of semisimple ones. The element $\sigma_2$ belongs to $H$ but not to $H_1$. In each coset of $K$ in $H_1$, one can find an element of the form $\sigma_2^{2m}\sigma_3^n$. It remains to see that $K$ itself can be generated by suitable combinations of $\sigma_2$ and $\sigma_3$. One has
\[\sigma_3^{-k}\sigma_2^{2m}\sigma_3^{-1}\sigma_2^{2n-2m}\sigma_3\sigma_2^{-2n}\sigma_3^{k} = 
    \begin{pmatrix}
    1 & 0 & 0 \\
    \left( x^{3m}-x^{3n} \right)\left( 1 - x^2 \right) x^k & 1 & 0 \\
    \left( x^{3m}-x^{3n} \right)\left( 1 - x^2 \right) x^{2k} & 0 & 1
    \end{pmatrix}.
\]
As $q\geq 5$, those elements are sufficient to generate $K$. We conclude that $\sigma_2$ and $\sigma_3$ generate $H$. Using the duality, we see that $\sigma_1$ and $\sigma_2$ generate the whole group of projective transformations which stabilise $\set{l,l'}$.

From this it is easy to see that $\left\langle \sigma_1,\ \sigma_2,\ \sigma_3 \right\rangle = \PSL(3,q)$.

It remains to prove that the polytope we created is a chiral one. For this, we need to prove that there does not exist an automorphism which sends $\sigma_1$ to its inverse, and which fixes $\sigma_1\sigma_2$ and $\sigma_1\sigma_2\sigma_3$. Note that our automorphism, if it exists, cannot exchange points and lines since $\left\langle\sigma_1,\ \sigma_2\right\rangle$ is the stabiliser of two lines. Therefore, our automorphism is given by a conjugation in $\PGaL(3,q)$. The involution $\sigma_1\sigma_2$ has center $(1,1+x^{-1},2)$ and axis $Z=0$ while the involution $\sigma_1\sigma_2\sigma_3$ has center $(1,1+x^{-1},1+x)$ and axis $(1-x)X=Z$. The projectivity $\sigma_1$ fixes three points namely $(1,0,0)$, $(1,1-x,1-x)$ and $(1,x^{-1}-x,1-x^2)$.

We decompose our automorphism as a part given by conjugation by $\phi \in \PGL(3,q)$ and a part induced by the field automorphism $\lambda \to \lambda^{p^i}$. We use firstly that $(0,1,0)$ is fixed (as the intersection of the two axes), and secondly that $(1,0,0)$ is fixed (it cannot be exchanged with the two other fixed points as $Z=0$ is fixed). We get that $\phi$ has the form
\[ \begin{pmatrix}
 a & 0& e \\ 0 & b & d \\ 0 & 0 & c
\end{pmatrix}. \]
We then use the fact that our automorphism fixes $\sigma_3$. The projectivity $\sigma_3$ has fixed points $(x,1+x,x)$, $(0,1,0)$ and $(0,0,1)$ with eigenvalues $x$, $1$ and $x^{-1}$ respectively (up to a scalar). The two points $(x,1+x,x)$ and $(0,0,1)$ could be exchanged by our automorphism only if $x^{p^i}=x^{-1}$, which is impossible by our choice of $x$. Therefore, they are both fixed and $d=e=0$ and $a=c$. Then it is impossible for our projectivity to send $\sigma_1$ to its inverse and the polytope has to be chiral.
\end{proof}
\section{Chiral polytopes of rank five for $\PSL(3,q)$}\label{srank5}
We give a constructive proof of Theorem~\ref{rank5}, that is, we give sets of 4-tuples of elements of $\PSL(3,q)$ that satisfy the conditions stated in Section~\ref{prelim}.

\begin{proof}[Proof of Theorem~\ref{rank5}]
    It is easy to check that the products of successive generators are involutions. One has to verify that the intersection property is satisfied. The generated subgroups are (where $\ast$ denotes an arbitrary element of $\FF_q$ and $j =0,1$ or $2$)
    \[
    \left\langle \sigma_1,\sigma_2 \right\rangle
    = \set{\begin{pmatrix}
        \pm 1 & 0 & \ast \\
        0 & \pm \omega^{ij} & \ast \\
        0 & 0 & \omega^{2ij}
    \end{pmatrix}},
    \left\langle \sigma_1,\sigma_2,\sigma_3 \right\rangle
    = \set{\begin{pmatrix}
        \pm 1 & \ast & \ast \\
        0 & \ast & \ast \\
        0 & \ast & \ast
    \end{pmatrix}} \]
    and
    $\left\langle \sigma_1,\sigma_2,\sigma_3,\sigma_4 \right\rangle = \PSL(3,q)$.
    For the other subgroups, one needs to distinguish depending on the value of $k$. Firstly, we consider the case $k = -1$. The subgroup $\left\langle \sigma_3,\sigma_4 \right\rangle$ preserves the four points $(1,0,0)$, $(0,1,0)$, $(0,0,1)$ and $(1,1,1)$. The generators induce cycles of order $3$ on those four points and generate together the alternating group $\Alt(4)$. Since the action on the four points is faithful, $\left\langle \sigma_3,\sigma_4 \right\rangle \cong \Alt(4)$. The subgroup $\left\langle \sigma_2,\sigma_3,\sigma_4 \right\rangle$ preserves the conic $XY + \omega YZ + \omega^2 ZX = 0$.
    
    Secondly, we consider the case $k = 1$. Then the subgroup $\left\langle \sigma_2,\sigma_3,\sigma_4 \right\rangle$ preserves the point $(1,\omega,\omega^2)$ while the subgroup $\left\langle \sigma_3,\sigma_4 \right\rangle$ also preserves the point $(1,\omega^2,\omega)$. We can see that the dual of the polytope with $(k,i) = (1,1)$ is the polytope with $(k,i) = (1,2)$. The isomorphism between the two is given by a conjugation by
    \[ \begin{pmatrix}
        \omega & \omega^2 & 1 \\ \omega^2 & \omega & 1 \\ 1 & 1 & 1
    \end{pmatrix}. \]
    For the intersection property, it remains to verify that  $\left\langle \sigma_1,\sigma_2,\sigma_3 \right\rangle \cap \left\langle \sigma_2,\sigma_3,\sigma_4 \right\rangle = \left\langle \sigma_2,\sigma_3\right\rangle$, i.e. that $\left\langle \sigma_2,\sigma_3 \right\rangle$ is the group of elements that preserve $(1,0,0)$ and $(1,\omega,\omega^2)$. This is a long but easy computation thus we omit the details here.
    
    Finally, there is no automorphism which sends $\sigma_1$ to $\sigma_1^{-1}$ and preserves the three involutions $\sigma_1\sigma_2$, $\sigma_1\sigma_2\sigma_3$ and $\sigma_1\sigma_2\sigma_3\sigma_4$. Indeed, every automorphism of $\PSL(3,q)$ is given by a conjugation in $\PGaL(3,q)$ or by a duality which exchanges points and lines. It cannot be a duality here since the stabiliser of a point would be transformed in the stabiliser of a line. Moreover, since $q = p$ is a prime, there is no field automorphism. Hence, the automorphism would be given by a conjugation by an element $\phi$ of $\PGL(3,q)$. Such a projectivity $\phi$ would need to preserve the axes and the centers of our three involutions. It needs to preserve the axes $Z=0$, $Y=Z$ and $X=Y$, and the centers $(1,1,-2)$, $(0,1,-1)$ and $(1,-1,0)$. This implies that $\phi$ is the identity, in which case it cannot conjugate $\sigma_1$ with its inverse and the corresponding polytope is necessarily chiral.
\end{proof}
This construction provides six chiral polytopes if we also count the duals of the polytopes of type $[3,6,3,3]$. Up to now, no other polytope of rank $5$ is known for the group $\PSL(3,q)$. The choice of $i =1$ or $2$ amounts to choose one primitive third root of unity or the other. Finally, if $q = p^2$ where $p$ is a prime number congruent to $5$ modulo $6$, this construction does not yield a chiral polytope but a directly regular one. Indeed, there is a field automorphism $\phi$ which sends $\sigma_1$ to its inverse and preserves the three involutions in this case.

\begin{comment}
In order to construct candidates for polytopes of rank $5$, the authors have used the following strategy. The generators $\sigma_1$ and $\sigma_4$ commute with each other and the involution $\sigma_1\sigma_2\sigma_3\sigma_4$ conjugate them with their inverse. Therefore, those three elements generate a group of the form $(C_m \times C_n) : C_2$. They constructed such a group with $m=n=3$ and they found involutions $\sigma_1\sigma_2$ which commute with $\sigma_1\sigma_2\sigma_3\sigma_4$. In the end comes the hard step of checking the intersection property. One can suppose
\[
    \sigma_1 = \begin{pmatrix}
    1 & 0 & 0 \\
    0 & \omega^i & 0 \\
    0 & 0 & \omega^{2i}
    \end{pmatrix},\ 
    \sigma_4 = \begin{pmatrix}
    0 & 1 & 0 \\
    0 & 0 & 1 \\
    1 & 0 & 0
    \end{pmatrix}
    \text{ and }
    \sigma_1\sigma_2\sigma_3\sigma_4 = \begin{pmatrix}
    0 & -1 & 0 \\
    -1 & 0 & 0 \\
    0 & 0 & -1
    \end{pmatrix}
    \]
Then, we find that the involution $\sigma_1\sigma_2$ is given by (where $2a(a+1)+bc=0$ )
\[
    \sigma_1\sigma_2 = \begin{pmatrix}
    a & a+1 & b \\
    a+1 & a & b \\
    c & c & -2a-1
    \end{pmatrix}
\]
Finally, in order for $\left\langle\sigma_1,\sigma_2,\sigma_3\right\rangle$ to be small enough, one has to assume it preserves some point or some line or some triangle or some conic or ... so that it is contained in one of the maximal subgroups. The same goes for $\left\langle\sigma_2,\sigma_3,\sigma_4\right\rangle$. This gives restrictions for $a$, $b$ and $c$.
\end{comment}

\section{Regular polytopes for subgroups of $\PSL(3,q)$}\label{regular}

Before proving Theorem~\ref{rank6}, we need to determine the directly regular polytopes of rank four that have a subgroup of $\PSL(3,q)$ as automorphism group. In order to do that, we first classify the regular polytopes of rank two for subgroups of $\PSL(3,q)$, then the regular polytopes of rank three and four for subgroups of $\PSL(3,q)$ with $q$ odd, and finally the regular polytopes of rank up to four for the subgroups of $\PSL(3,q)$ with $q$ even, dividing that into the case where the subgroups are triangular subgroups and the case where they are not.

\subsection{Regular polytopes of rank $2$}
Rank two polytopes have automorphism groups that are dihedral groups.

We list the dihedral subgroups in $\PSL(3,q)$ up to isomorphism. In each case we give a generator $\sigma$ of the cyclic subgroup of index $2$ and an involution $\tau$ which, together with $\sigma$, generates the dihedral subgroup.

There are two cases where we get a dihedral group of order $4$, which is abelian, namely

\[\setlength{\arraycolsep}{3pt}\begin{array}{ll}
 \textbf{Case A : }
\sigma = \begin{pmatrix}
    1 & 1 & 0 \\
    0 & 1 & 0 \\
    0 & 0 & 1
\end{pmatrix}  & \text{and }
\tau = \begin{pmatrix}
    1 & a & 0 \\
    0 & 1 & 0 \\
    0 & 0 & 1
\end{pmatrix} \text{ in even characteristic.}
\end{array}
\]

\[\setlength{\arraycolsep}{3pt}\begin{array}{ll}
\textbf{Case B : }
\sigma = \begin{pmatrix}
    -1 & 0 & 0 \\
     0 & -1 & 0 \\
     0 &  0 & 1
\end{pmatrix}  & \text{and }
\tau = \begin{pmatrix}
    1 & 0 & 0 \\
    0 & -1 & 0 \\
    0 &  0 & -1
\end{pmatrix} \text{ in odd characteristic.}\\
\end{array}
\]

In the non-abelian cases, we get the following.

\[\setlength{\arraycolsep}{3pt}
\begin{array}{ll}
   \textbf{Case 1 : }
\sigma = \begin{pmatrix}
    x & 0 & 0 \\
    0 & 1 & 0 \\
    0 & 0 & x^{-1}
\end{pmatrix}& \text{and }
\tau = \begin{pmatrix}
    0 & 0 & 1 \\
    0 & -1 & 0 \\
    1 & 0 & 0
\end{pmatrix}
 
\begin{matrix}
   \text{ with } x \text{ an element of } \FF_q \\
   \text{ different from }0, \pm 1.
\end{matrix}
\end{array}
\]

\hspace{-2cm}\[\setlength{\arraycolsep}{3pt}
\begin{array}{ll}
   \textbf{Case 2 : }
\sigma = \begin{pmatrix}
    1 & 1 & 0 \\
    0 & 1 & 1 \\
    0 & 0 & 1
\end{pmatrix}& \text{and }
\tau = \begin{pmatrix}
    -1 & -1 & 0\\
     0 & 1 & 0 \\
    0 & 0 & -1
\end{pmatrix}.
\end{array}
\]

\[\setlength{\arraycolsep}{3pt}
\begin{array}{ll}
 \textbf{Case 3 : }
\sigma = \begin{pmatrix}
    \pm 1 & 1 & 0 \\
    0 & \pm 1 & 0 \\
    0 & 0 & 1
\end{pmatrix}& \text{and }
\tau = \begin{pmatrix}
    1 & 0 & 0 \\
    0 & -1 & 0 \\
    0 & 0 & -1
\end{pmatrix} \text{ in odd characteristic.}
\end{array}
\]

\[\setlength{\arraycolsep}{3pt}
\begin{array}{ll}
 \textbf{Case 4 : }
\sigma = \begin{pmatrix}
    1 & 0 & 0 \\
    0 & 0 & -1 \\
    0 & 1 & x
\end{pmatrix}&  \text{and }
\tau = \begin{pmatrix}
    -1 & 0 & 0\\
     0 & 1 & x \\
    0 & 0 & -1
\end{pmatrix}\begin{matrix}
    \text{ with } x \text{ an element of }\FF_q\\
    \text{ such that }X^2 - xX + 1 \\
    \text{ is irreducible.}
\end{matrix}
\end{array}
\]

The element $\sigma$ has at most order $q-1$, $2\cdot GCD(2,p)$, $2p$ and $q+1$ respectively. Hence, the order of any dihedral subgroup is a divisor of $2(q-1)$, $2(q+1)$ or $4p$. 

%{\color{red}How did you get this list? From which classification theorem? Mitchell???}{\color{OliveGreen}I obtained this list through computations. The fact $\sigma$ is conjugated to its inverse easily restricts the form of $\sigma$. The we put $\sigma$ in canonical form. We know that the involution $\tau$ preserves eigenspaces, ... of $\sigma$. This considerably restricts the form of $\tau$ together with the fact that it is an involution. In the end, with some base change we can assume $\tau$ has those nice form. I did not write all the proof since I assume this simple result is already known somewhere.}
%We want a geometric description of those subgroups.
\subsection{Regular polytope of rank $3$ for $q$ odd}

We want to search for involutions $\rho_1$, $\rho_2$ and $\rho_3\in \PSL(3,q)$ such that $(\langle \rho_1,\rho_2,\rho_3\rangle,\{ \rho_1,\rho_2,\rho_3\})$ is a string C-group. Let $G=\langle \rho_1,\rho_2,\rho_3\rangle$. The parabolic subgroup $G_2 := \langle \rho_1,\rho_3\rangle$ preserves a triangle in $\PG(2,q)$. As the group $\PGL(3,q)$ acts transitively on the triangles of $\PG(2,q)$ and $G\leq \PGL(3,q)$, we can choose homogeneous coordinates in $\PG(2,q)$ so that the vertices of this triangle are the points with respective coordinates $(1,0,0)$, $(0,1,0)$ and $(0,0,1)$. Hence, $G_2$ consists of the four transformations below.
\[ 
\begin{pmatrix} \pm1 & 0 & 0 \\ 0 & \pm1 & 0 \\ 0 & 0 & \pm1 \end{pmatrix}.
\]
Among those four transformations, we denote by $\phi_1$, $\phi_2$ and $\phi_3$, the involutions respectively with center $(1,0,0)$, $(0,1,0)$ and $(0:0:1)$.

If $\rho_2$ does not stabilise an edge nor a vertex, then, by~\cite{BV2010}, the group $\Gamma$ preserves a bilinear form. We leave this case for later and we consider the case where $\rho_2$ stabilises one of the two. Then either the axis of $\rho_2$ contains a vertex or its center is contained in an edge. Using the duality between lines and points in the projective plane, we can assume that the axis contains at least one vertex. We prove the following result.

\begin{theorem}
    Let $G := \langle \rho_1, \rho_2,\rho_3 \rangle \leq \PSL(3,q)$  be such that $(G,\{ \rho_1, \rho_2,\rho_3\})$ is a string $C$-group of rank $3$.  Suppose that $G$ stabilises a point or a line of $\PG(2,q)$. Then $G$ is one of the following subgroups
    \[ D_{4p}, D_{2k}, D_{2p}^2, D_{2p} \wr C_2, He_p : C_2^2 \text{ or }
    E_{q'}^2 : D_{2k} \]
    where $q'$ is the smallest root of $q$ such that $k$ divides $q' \pm 1$.
\end{theorem}
\begin{proof}
If the axis of $\rho_2$ contains two vertices, it is an edge of the triangle. Then its center can either be the third vertex, or a point of an edge, or a point outside of the triangle. Those three situations correspond to the following matrices for $\rho_2$
\[
\begin{pmatrix} 1 & 0 & 0 \\ 0 & -1 & 0 \\ 0 & 0 & -1 \end{pmatrix} \text{ or }
\begin{pmatrix} 1 & 0 & 0 \\ 1 & -1 & 0 \\ 0 & 0 & -1 \end{pmatrix} \text{ or }
\begin{pmatrix} 1 & 0 & 0 \\ 1 & -1 & 0 \\ 1 & 0 & -1 \end{pmatrix}.
\]

If the axis of $\rho_2$ contains exactly one vertex, then we  call the lines which contain this vertex adjacent lines and the line which does not the opposite line. We can look at those three lines through the vertex (i.e. the axis and the two adjacent edges). There is a fourth line through this vertex which forms a harmonic quadruple together with them. We have two choices to make to place the center. Firstly, we choose whether or not it is on the opposite line. Secondly, we choose if the center belongs to the harmonic conjugate, or to an adjacent line or to none of the preceding. There are two possibilities for the first choice and three for the second. It yields a total of six possibilities. Those six possibilities are given below.
\[
\begin{pmatrix} 1 & -2 & 0 \\ 0 & -1 & 0 \\ 0 & 0 & -1 \end{pmatrix},
\begin{pmatrix} 0 & -1 & 0 \\ -1 & 0 & 0 \\ 0 & 0 & -1 \end{pmatrix},
\begin{pmatrix} a & -a-1 & 0 \\ a-1 & -a & 0 \\ 0 & 0 & -1 \end{pmatrix},\]\[
\begin{pmatrix} 1 & -2 & 0 \\ 0 & -1 & 0 \\ 1 & -1 & -1 \end{pmatrix},
\begin{pmatrix} 0 & -1 & 0 \\ -1 & 0 & 0 \\ 1 & -1 & -1 \end{pmatrix},
\begin{pmatrix} a & -a-1 & 0 \\ a-1 & -a & 0 \\ 1 & -1 & -1 \end{pmatrix}.
\]

Hence, we have found nine cases for $\rho_2$. Note however that the second and the fourth cases are dual. Therefore, we can remove the fourth and there remain eight cases. We now show that the subgroups they generate are respectively $D_4$, $D_{4p}$, $D_{2p}^2$, $D_8$, $D_{2k}$, $He_p : C_2^2$, $D_{2p} \wr C_2$ and $E_{q'}^2 : D_{2k}$.

In the first case, $\rho_2$ is in $G_2$ which is impossible.

In the second case, 
\[ G = \set{ \begin{pmatrix}
    \pm 1 & 0 & 0 \\ a & \pm 1 & 0 \\ 0 & 0 & \pm 1
\end{pmatrix} \middle| a \in \FF_p} \cong D_{4p} \]

In the third case, 
\[ G = \set{ \begin{pmatrix}
    \pm 1 & 0 & 0 \\ a & \pm 1 & 0 \\ b & 0 & \pm 1
\end{pmatrix} \middle| a,b \in \FF_p} \cong D_{2p}^2 \]

In the fourth case, it is not possible to choose $\rho_1$ and $\rho_3$ in $G_2$ and to respect the intersection property $IP$. Indeed, we have $\phi_1\rho_2\phi_1 = \phi_2\rho_2\phi_2 = \phi_3\rho_2$.

 In the fifth case, the involution $\phi_3$ is in the center of $G$. Note that either $\phi_1\rho_2$ or $\phi_2\rho_2$ has even order. Indeed, by computing the characteristic polynomials, one sees that, if $\phi_1\rho_2$ has eigenvalues $1$, $\lambda$ and $\lambda^{-1}$ ($\lambda \not= \pm 1$), then $\phi_2\rho_2$ has eigenvalues $1$, $-\lambda$ and $-\lambda^{-1}$. Without loss of generality, we suppose that $\phi_1\rho_2$ has even order $2k$. Then, the dihedral subgroup $D_{4k}$ generated by $\phi_1$ and $\rho_2$ contains $\phi_3$ (and $\phi_2$) since $(\phi_1\rho_2)^k= \phi_3$. Therefore, the generated subgroup $\Gamma$ is dihedral.
Dihedral subgroups representations as string $C$-groups were investigated in~\cite{ConnorLeemans_Suzuki} where it is shown that the only possible string $C$-group structure of rank $3$ appears for dihedral subgroups $D_{4k}$ with $k$ odd and yields Schl\"afli type $[k,2]$.

In the sixth case, we can replace $\rho_2$ by 
\[\begin{pmatrix}
    -1 & 0 & 0 \\ 2 & 1 & 0 \\ 2 & 2 & -1
\end{pmatrix}\]
using a conjugation. With this change, we have
\[ \Gamma = \set{ \begin{pmatrix}
    \pm 1 & 0 & 0 \\ a & \pm 1 & 0 \\ b & c & \pm 1
\end{pmatrix} \middle| a,b,c \in \FF_p} \cong He_p : C_2^2 \]

In the seventh case,
\[ \Gamma = \set{ \begin{pmatrix}
    \pm 1 & 0 & 0 \\ 0 & \pm 1 & 0 \\ a & b & \pm 1
\end{pmatrix},\begin{pmatrix}
    0 & \pm 1 & 0 \\ \pm 1 & 0 & 0 \\ a & b & \pm 1
\end{pmatrix}  \middle| a,b \in \FF_p} \cong D_{2p} \wr C_2 \]

The eight case will need a deeper examination. We want to prove that the generated subgroup is $E_{q'}^2 : D_{2k}$. We start by constructing the elementary abelian group $E_{q'}^2$. We let
\[  M(x,y) = \begin{pmatrix}
    1 & 0 & 0 \\ 0 & 1 & 0 \\ x & y & 1
\end{pmatrix}. \]
Using this notation, we have
\[  M(x,y)M(x',y') = M(x+x',y+y'),\]
\[M(x,y)^{\phi_1} = M(-x,y),\] 
\[M(x,y)^{\phi_2} = M(x,-y),\]
\[M(x,y)^{\phi_3} = M(-x,-y), \]
and finally
\[ M(x,y)^{\rho_2} = M((-a)x+(-a+1)y,(a+1)x+ay). \]
We can therefore deduce that the set of matrices $M(x,y)$ form a subgroup $H$ of $\PSL(3,q)$ which is stable under conjugation by $G$. We  
now determine which subgroups of $H$ are also stable under this conjugation. Let $K$ be such a subgroup and let us suppose that $M(x,y) \in K$. Then,
\[ M(2x,0) = M(x,y)\phi_2M(x,y)\phi_2 \in K \text{ and } M(x,0) = M(2x,0)^\frac{p+1}{2} \in K.\]
Therefore, 
\[ M( (-a)x, (a+1)x) = \rho_2M(x,0)\rho_2 \in K \text{ and } M(ax,0) \in K. \]
Iterating this argument, one can see that $M(a^ix,0)$ is in $K$ and using a similar argument that $M(0,a^iy)$, $M(a^i(a-1)y,0)$ and $M(0,a^i(a+1)x)$ are also in $K$ for all non-negative integers $i$. By adding such terms, one can get any term of the form
\[ M( (k_0 + k_1a + \cdots + k_na^n)x + (k_0' + k_1'a + \cdots + k_n'a^n)(a-1)y,\]\[(l_0' + l_1'a + + \cdots + l_m'a^m)(a+1)x + (l_0 + l_1a + + \cdots + l_ma^m)y ) \in K \]
where the $k_i$'s and the $l_i$'s belong to $\FF_p$. If we define $\FF_{q'}$ to be the subfield of $\FF_q$ generated by $a$, then we can get any term of the form $M(kx + k'(a-1)y,l'(a+1) + ly)$ where $k$, $k'$, $l$ and $l'$ belongs to $\FF_{q'}$. One can easily see that by using any kind of manipulation, we will still get a term of this form.

Now, we can see that $M(2,-2) = (\rho_2\phi_3)^2$ belongs to $\Gamma$. Henceforth, $\Gamma$ contains the following elementary abelian normal subgroup
\[ E_{q'}^2 = \set{ \begin{pmatrix}
    1 & 0 & 0 \\ 0 & 1 & 0 \\ x & y & 1
\end{pmatrix} \middle| x,y \in \FF_{q'}}. \]
Moreover, it also contains
\[
M(-1,1)\rho_2 = \begin{pmatrix} a & -a-1 & 0 \\ a-1 & -a & 0 \\ 0 & 0 & -1 \end{pmatrix}.
\]
Thus, it contains the subgroup $D_{2k}$ from the fifth case. Therefore, we can conclude that $E_{q'}^2 : D_{2k}$ is contained in $\Gamma$ and since it already contains each generators of $\Gamma$, it is $\Gamma$.
\end{proof}

We now consider the case where $\rho_2$ does not preserve an edge or a vertex of the triangle.

\begin{theorem}\label{thm:5.2}
    Let $G := \langle \rho_1, \rho_2,\rho_3 \rangle \leq \PSL(3,q)$  be such that $(G,\{ \rho_1, \rho_2,\rho_3\})$ is a string $C$-group of rank $3$.  Suppose that $G$ does not stabilise a point or a line. Then $G$ preserves a conic and is one of the following subgroups where $q'$ is a root of $q$.
    \[ \Sym(4), \Alt(5), \POmega(3,q'), \PSO(3,q') \]
\end{theorem}
\begin{proof}
Without loss of generality, we can suppose that $\rho_1 = \phi_1$, $\rho_3 = \phi_3$ and the axis of $\rho_2$ is $x_1-x_2+x_3=0$. With those assumptions,
\[ \rho_2 = \begin{pmatrix} a & -a-1 & a+1 \\ a+b & -a-b-1 & a+b \\
    b+1 & -b-1 & b \end{pmatrix} \]
where $a +1$, $a + b$ and $b+1$ are nonzero. Then the following quadratic form is preserved
\[ Q(x_1,x_2,x_3) = \frac{x_1^2}{a+1} - \frac{x_2^2}{a+b} + \frac{x_3^2}{b+1} \]

There is a well known isomorphism between $\PSO(3,q)$ and $\PGL(2,q)$ (see for instance~\cite[Page 125]{Wilson}).
Under this isomorphism, a subgroup which fixes a point in $\PGL(2,q)$ is sent to a subgroup which fixes a point in $\PSO(3,q)$. We are only interested in subgroups which do not fix a point or a line in $\PSO(3,q)$. Therefore the generated subgroup can only be isomorphic to one of $C_k$, $D_k$, $\Alt(4)$, $\Sym(4)$, $\Alt(5)$, $\PSL(2,q')$ or $\PGL(2,q')$ (with $k|q\pm1$ and $q'$ a root of $q$) since all the other subgroups of $\PGL(2,q)$ fix a point. Moreover, it cannot be $C_k$ nor $D_k$ since those subgroups fix a point in $\PSL(3,q)$. There remain only $\Alt(4)$, $\Sym(4)$, $\Alt(5)$, $\PSL(2,q')$ or $\PGL(2,q')$.

First, we examine the cases where the generated subgroup is an exceptional subgroup of $\PGL(2,q)$ (i.e. $\Alt(4)$, $\Sym(4)$ or $\Alt(5)$). It cannot be $\Alt(4)$ since $\Alt(4)$ is not generated by involutions. It can be $\Sym(4)$ only if $\rho_1\rho_2$ and $\rho_2\rho_3$ have orders $3$ or $4$. It can be $\Alt(5)$ only if $\rho_1\rho_2$ and $\rho_2\rho_3$ have orders $3$ or $5$.

The groups  $\POmega(3,q')$ and $\PSO(3,q')$ correspond to the two remaining cases $\PSL(2,q')$ and $\PGL(2,q')$ which cannot be eliminated.
This concludes the proof of the Theorem.
\end{proof}

Observe that the regular polytopes of rank three with automorphism group $\PSL(2,q)$ or $\PGL(2,q)$ have been classified in~\cite{CPS2008}.

Let us examine how the generated subgroup depends on $a$ and $b$.

Note that the characteristic polynomial of $\rho_1\rho_2$ is $(X-1)(X^2 - 2a + 1)$ while the polynomial of $\rho_2\rho_3$ is $(X-1)(X^2 - 2b + 1)$. Hence, the order of $\rho_1\rho_2$ can be $k$ only when $a = (r+r^{-1})/2$ for $r$ a $k$-root of unity. Therefore, $\rho_1\rho_2$ has order $3$ when $a=-1/2$. It has order $4$ when $a = 0$. It has order $5$, when $2a$ is a sum of two inverse fifth roots of unity. This happens when $4a^2 + 2a - 1 = 0$.

It is possible to check that when $(a,b) = (-1/2,-1/2)$, $(0,-1/2)$ or $(-1/2,0)$ we get the subgroup $\Sym(4)$. If $X$ is a square root of $5$ then, we get $\Alt(5)$ when $(a,b) = (-1/2,(-1+X)/4)$, $(-1/2,(-1-X)/4)$, $((-1+X)/4,(-1-X)/4)$ or the same with $a$ and $b$ exchanged.

\subsection{Regular polytope of rank $4$ for $q$ odd}\label{qodd}

We now determine which subgroups of $\PSL(3,q)$ with $q$ odd are string C-groups of rank four.
\begin{theorem}
Let $S:=\{\rho_1,\rho_2,\rho_3,\rho_4\}$ be a set of involutions of $\PSL(3,q)$ ($q$ odd) and let $G := \langle S \rangle$. If the pair $(G,S)$ is a string $C$-group of rank $4$, then $G$ is one of the following
    \begin{itemize}
        %\item A dihedral group $D_{2k}$ where $k$ divides $q-1$, $q+1$ or $2p$,
        \item $D_{2p}^2$ which preserves two points,
        \item $He_p : C_2^2$ which preserves a point and a line (adjacent) or
        %\item $D_{2p}\text{ wr }C_2$ which preserves a set of two points,
        %\item $E_{q'}^2 : D_{2k}$ which preserves a point,
        \item $\POmega(3,q')$ or $\PSO(3,q')$ which preserves a conic.
        %\item An exceptional subgroup $\Sym(4)$ or $\Alt(5)$.
    \end{itemize}
\end{theorem}

\begin{proof}
We can choose the coordinates so that $\rho_1$ and $\rho_4$ are diagonal matrices. Then, using the fact that $\rho_3$ commutes with $\rho_1$ and $\rho_2$ commutes with $\rho_4$, we find that
\[\rho_1 = \begin{pmatrix} 1 & 0 & 0 \\ 0 & -1 & 0 \\ 0 & 0 & -1 \end{pmatrix},\ 
\rho_2 = \begin{pmatrix} a & b & 0 \\ c & -a & 0 \\ 0 & 0 & -1 \end{pmatrix},\]\[
\rho_3 = \begin{pmatrix} -1 & 0 & 0 \\ 0 & -a' & c' \\ 0 & b' & a' \end{pmatrix},\ 
\rho_4 = \begin{pmatrix} -1 & 0 & 0 \\ 0 & -1 & 0 \\ 0 & 0 & 1 \end{pmatrix} \]
with $1 - a^2 = bc$ and $1 - a'^2 = b'c'$. We start by noticing some restrictions on the parameters appearing in the matrices. If $a = 0$ or $a = -1$, then $\left\langle \rho_1,\rho_2,\rho_4\right\rangle$ is not a string $C$-group. Hence, we can suppose that $a$ is not $0$ nor $-1$.  Then, $b=0$ or $c=0$ if and only if $a=1$. Moreover, $b$ and $c$ cannot be $0$ at the same time. We have a similar result for $a'$, $b'$ and $c'$. We distinguish four situations (depending on whether $b$, $b'$, $c$ and $c'$ are zero's) :
\begin{enumerate}
    \item if none is zero, we can re-scale the coordinates to get $b=-a-1$, $c = a-1$, $b' = -a'-1$ and $c' = a'-1$.
    \item if one is zero, we suppose $b=0$ and we can re-scale the coordinates to get $c=1$, $b'=-a'-1$ and $c'=a'-1$.
    \item if two are zero's, we suppose $b=b'=0$ and we can re-scale the coordinates to get $c=c'=1$.
    \item if two are zero's, we suppose $b=c'=0$ and we can re-scale the coordinates to get $b'=c=1$. 
\end{enumerate}

The third and the fourth cases yield respectively the polytopes $D_{2p}$ and $He_p : C_2^2$. In the second case, $G_1$ is the subgroup $E_{q'}^2 : D_{2k}$ and it contains $\rho_1$, which is impossible for a string $C$-group. It remains to consider the first case.

The subgroup $G$ generated by the involutions leaves the quadratic form
\[ Q(x_1,x_2,x_3) = (1-a)(1+a')x_1^2 + (1+a)(1+a')x_2^2 + (1+a)(1-a')x_3^2 \] invariant.
Therefore, $G$ is contained in $\PSO(3,q) \cong \PGL(2,q)$. We identify the parabolic subgroups $G_1$ and $G_4$ using the classification we obtained for rank $3$. They fall into the category of string $C$-groups which do not preserve any point nor any line. Hence, they can be of the form $\Sym(4)$, $\Alt(5)$, $\POmega(3,q'')$ or $\PSO(3,q'')$ by Theorem~\ref{thm:5.2}.

If either $G_1$ or $G_4$ is a group $\POmega(3,q'')$ or $\PSO(3,q'')$, then $G$ is also of the form $\POmega(3,q')$ or $\PSO(3,q')$ and those have be investigated in~\cite{ls07} and~\cite{ls09} respectively.

If both $G_1$ and $G_4$ are of the form $\Sym(4)$ or $\Alt(5)$, then we have restrictions on the orders of $\rho_1\rho_2$, $\rho_2\rho_3$ and $\rho_3\rho_4$. 
The characteristic polynomials of those three products are
\[ (X-1)(X^2 -2a X +1),\ (X-1)(X^2 + (-aa' +a + a' + 1)X + 1)\text{ and } (X-1)(X^2 - 2bX + 1). \]
We summarize the different possibilities for $a$ and $a'$ in Table~\ref{poss}. The table must be read as follow. At the left, one finds the values of $a$ for which $\rho_1\rho_2$ has order $3$ or $5$. At the right, one finds the values of $a'$ for which $\rho_3\rho_4$ has order $3$ or $5$. Given the value of $a$ and $a'$, one can compute $-aa'+a+a'+1$. Finally, we have some restrictions on $-aa'+a+a'+1$ for the product $\rho_2\rho_3$ to have order $3$, $4$ or $5$. For example, in the case $[3,3,3]$, we need to have $-1/4 = -aa'+a+a'+1 = 1$ which can only happen when $p=5$.
\end{proof}

\begin{table}
\begin{center}
\begin{tabular}{l|c|c|c|c|r}
    & $-aa'+a+a'+1$ & $[-,3,-]$ & $[-,4,-]$ & $[-,5,-]$ & \\ \hline
    $[3,-,-]$,
    & $\frac{-1}{4}$ & $p=5$ & No & $p=11$ & $a' = \frac{-1}{2},\ [-,-,3]$ \\
    $a = \frac{-1}{2}$& $\frac{1+3X}{8}$ & No & - & $p=59$ & $a' = \frac{-1+X}{4},\ [-,-,5]$ \\
    & $\frac{1-3X}{8}$ & No & - & $p=59$ & $a' = \frac{-1-X}{4},\ [-,-,5]$ \\ \hline
    $[5,-,-],$
    & $\frac{1+3X}{8}$ & No & - & $p=59$ & $a' = \frac{-1}{2},\ [-,-,3]$ \\
    $a = \frac{-1+X}{4}$& $\frac{1+5X}{8}$ & $p=19$ & - & $p=3,11$ & $a' = \frac{-1+X}{4},\ [-,-,5]$ \\
    & $\frac{3}{4}$ & No & - & $p=19$ & $a' = \frac{-1-X}{4},\ [-,-,5]$ \\ \hline
    $[5,-,-],$
    & $\frac{1-3X}{8}$ & No & - & $p=59$ & $a' = \frac{-1}{2},\ [-,-,3]$ \\
    $ a = \frac{-1-X}{4}$& $\frac{3}{4}$ & No & - & $p=19$ & $a' = \frac{-1+X}{4},\ [-,-,5]$ \\
    & $\frac{1-5X}{8}$ & $p=19$ & - & $p=3,11$ & $a' = \frac{-1-X}{4},\ [-,-,5]$ \\ \hline
\end{tabular}
\caption{String C-groups of rank 4 for subgroups of $\PSL(3,q)$ ($q$ odd)}\label{poss}
\end{center}
\end{table}
\subsection{Regular polytope for $q$ even : triangular subgroups}

We start by considering the case where the involutions fix a point and an adjacent line.

\begin{theorem}
    Let $\Gamma := (G,S)$  be a non-degenerate string $C$-group with $G\leq \PSL(3,q)$ ($q$ even). Suppose that $G$ fixes a point and an adjacent line. Then $G$ is one of the following.
    \[C_2, D_8, C_2^2 \wr C_2 \]
    Moreover, the corresponding string C-groups are respectively of rank $1$, $2$ and $3$ and have respective Schl\"afli types $[]$, $[4]$ and $[4,4]$.
\end{theorem}

\begin{proof}
If $G$ fixes a point and an adjacent line, we have that $G$ is a subgroup of the $2$-group
\[\def\arraystretch{1}
\set{\begin{pmatrix}
1 & x & y \\
0  & 1 & z \\
0  & 0  & 1
\end{pmatrix} \middle| x,y,z \in \FF_q}. \]
Let $p$ and $l$ respectively be the preserved point and line. We start by constructing non-degenerate regular polytopes.

In rank $1$, the only regular polytope is the polytope of the group $C_2$. We can take any involution with axis $l$ or with center $p$.

In rank $2$, the $2$-group of unitriangular matrices contains only elements of order $2$ and $4$. The two involutions cannot have the same center nor can they have the same axis. Otherwise they would commute and give a degenerate polytope. Therefore, one of the involutions has center $p$ and axis different from $l$, and the other involution has center different from $p$ and axis $l$. One sees easily that they generate a dihedral subgroup of order $8$.

In rank $3$, we can suppose, using the duality, that $\rho_1$ has axis $l$. Then $\rho_2$ has center $p$ and $\rho_3$ has axis $l$. The axis of $\rho_2$ cannot be $l$. Similarly, the centers of $\rho_1$ and $\rho_3$ cannot be $p$. They generate the group $C_2 \times D_8$ or $C_2^2 \wr C_2$ depending on whether $\rho_1\rho_3$ commute with $\rho_2$ or not. Note however that $C_2 \times D_8$ does not verify the intersection property.

In rank $r\geq 4$, we can suppose that $\rho_1$ has axis $l$. Then for every odd $i\leq r$, the involution $\rho_{i}$ also has axis $l$ while for every even $j\leq r$, the involution $\rho_j$ has center $p$. Moreover, no involution can have at the same time axis $l$ and center $p$. Therefore, $\rho_1$ and $\rho_4$ cannot commute and we do not get a string $C$-group. Hence there is no such string C-group of rank at least four.

To summarize the above, we found three non-degenerate polytopes with groups $C_2$, $D_8$ and $C_2^2 \text{ wr }C_2$, and with Schl\"afli type $[]$, $[4]$ and $[4,4]$.
\end{proof}

\subsection{Regular polytope for $q$ even : other subgroups}\label{qeven}
We now take care of the case where there is no point and adjacent line which are fixed by $G$. First, we deal with rank $3$.
\begin{theorem}
   Let $S := \{\rho_1, \rho_2, \rho_3\}$ be a set of three involutions of $\PSL(3,q)$  ($q$ even). Let $G := \langle S \rangle$.
    Suppose $(G,S)$ is a string $C$-group of rank $3$. Suppose moreover that $G$ does not stabilise a point and an adjacent line of $\PG(2,q)$. Then $G$ is one of the following.
    \begin{itemize}
        \item $E_{q'}^2:D_{2k}$ where $q'$ is the smallest root of $q$ such that $k|q'\pm 1$;
        \item $\PSO(3,q')$ where $q' \mid q$.
    \end{itemize}
\end{theorem}

Observe that since here $q$ is even, $\PSO(3,q') \cong \POmega(3,q') \cong \PSL(2,q')$.
\begin{proof}
We can assume $\rho_1$ and $\rho_3$ have the same axis. Moreover, the axis of $\rho_2$ will intersect their common axis in one point $p$. Therefore, the point $p$ is fixed and no line through $p$ can be fixed. The point $p$ cannot be the center of $\rho_2$. Otherwise, the axis of $\rho_1$ would be fixed. Similarly, the two centers of $\rho_1$ and $\rho_3$ cannot both be the point $p$. Therefore, we can assume without loss of generality that the center of $\rho_1$ is not $p$. Depending on the position of the center of $\rho_3$, there are three possible situations which are depicted below. Points and lines correspond respectively to the centers and the axes of the involutions.\\

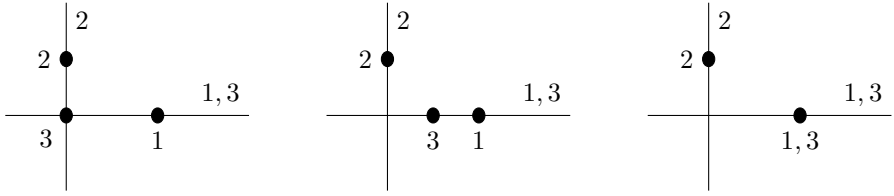
\begin{tikzpicture}[xscale=0.81]
    \draw (-1,0) -- (3,0) node[anchor=south east] {$1,3$} ;
    \draw (0,-1) -- (0,1.5) node[anchor=north west] {$2$} ;
    \draw[fill] (0,0) circle (0.1) ++(225:0.1) node[anchor=north east]{$3$} ;
    \draw[fill] (1.5,0) circle (0.1) ++(270:0.1) node[anchor=north]{$1$} ;
    \draw[fill] (0,0.75) circle (0.1) ++(180:0.1) node[anchor=east]{$2$} ;
    
    \begin{scope}[xshift=150]
    \draw (-1,0) -- (3,0) node[anchor=south east] {$1,3$} ;
    \draw (0,-1) -- (0,1.5) node[anchor=north west] {$2$} ;
    \draw[fill] (0.75,0) circle (0.1) ++(270:0.1) node[anchor=north]{$3$} ;
    \draw[fill] (1.5,0) circle (0.1) ++(270:0.1) node[anchor=north]{$1$} ;
    \draw[fill] (0,0.75) circle (0.1) ++(180:0.1) node[anchor=east]{$2$} ;
    \end{scope}
    
    \begin{scope}[xshift=300]
    \draw (-1,0) -- (3,0) node[anchor=south east] {$1,3$} ;
    \draw (0,-1) -- (0,1.5) node[anchor=north west] {$2$} ;
    \draw[fill] (1.5,0) circle (0.1) ++(270:0.1) node[anchor=north]{$1,3$} ;
    \draw[fill] (0,0.75) circle (0.1) ++(180:0.1) node[anchor=east]{$2$} ;
    \end{scope}
\end{tikzpicture}

We let the axis of $\rho_1$ be the line $x_3=0$ and we let the axis of $\rho_2$ be the line $x_1=0$. Finally, we let the line passing through the centers of $\rho_1$ and $\rho_2$ be $x_2=0$. We can then assume that
\[ \def\arraystretch{1}
\rho_1 = \begin{pmatrix} 
    1 & 0 & a \\ 0 & 1 & 0 \\ 0 & 0 & 1
\end{pmatrix},\ 
\rho_2 = \begin{pmatrix} 
    1 & 0 & 0 \\ 0 & 1 & 0 \\ 1 & 0 & 1
\end{pmatrix}
\text{ and }
\rho_3 = \begin{pmatrix} 
    1 & 0 & b \\ 0 & 1 & c \\ 0 & 0 & 1
\end{pmatrix}\]
Note that the subgroup of matrices
\[ \def\arraystretch{1}
M(x,y) = \begin{pmatrix}
    1 & 0 & 0 \\
    x & 1 & y \\
    0 & 0 & 1 \\
\end{pmatrix}.\]
is invariant under conjugation by elements of $G$. Indeed, we have
\[ M(x,y)^{\rho_1} = M(x,ax+y),\ 
M(x,y)^{\rho_2} = M(x+y,y) \text{ and }
M(x,y)^{\rho_3} = M(x,bx+y)\]
Using similar argument to those used with the odd characteristic, we see that, by using only products and conjugation, we can get from an element $M(x,y)$ all the elements of the form
\[ M(kx + k'y,l'x + ly) \text{ with } k,k',l,l' \in \FF_{q'} = \FF_2\left[ a,b \right]. \]
If $b=0$, then $\rho_3 = M(0,c)$ and we get the group $E_{q'}^2:D_{2k}$. If $a=b$, then $\rho_1\rho_3 = M(0,c)$ and we have the same conclusion. We can now suppose that $b$ is different from $0$ and $a$. We will see that, in those cases, we get the group $\PSL(2,q')$. When $c=0$, the line $x_2=0$ is fixed. When $c \not=0$, the following quadratic form is invariant under the action of $G$.
\[ Q(x,y,z) = \frac{x^2 + axz + az^2}{(a+b)b} + \frac{y^2}{c^2} \]
Independently of $c$, we find that the generated subgroup is contained in $\PSL(2,q)$. Finally, a careful but not difficult analysis of the subgroups of $\PSL(2,q)$ shows that the generated subgroup is $\PSL(2,q')$ where $q'$ is a root of $q$.
\end{proof}

The higher the rank, the more commutation constraints we have. In odd characteristic, this prevents from constructing a polytope of rank $5$ or higher. In even characteristic, it is possible to have a lot of involutions commuting together. Indeed, $\PSL(3,2^n)$ contains subgroups $C_2^{k}$ which forms a polytope of rank $k$ for $k \leq 2n$. However, we will see that, from rank $5$ onward, the only possible polytope for a subgroup of $\PSL(3,q)$ arise from triangular subgroups.

\begin{theorem}
Let $S:=\{\rho_1, \rho_2,\cdots,\rho_d\}$ be a set of involutions of $\PSL(3,q)$ ($q$ even). Let $G := \langle S \rangle$.
    Suppose that $(G,S)$ is a string $C$-group of rank $d\geq 4$.
    Suppose moreover that $G$ does not fix a point and an adjacent line. Then $d=4$, $(G,S)$ has Schl\"afli type $[4,k,4]$, $G_1 \cong G_4 \cong E_{q'}^2 : D_{2k}$ and $G \cong E_{q'}^4 : D_{2k}$.
\end{theorem}

\begin{proof}
The involutions of odd index, $\rho_1$, $\rho_3$, $\rho_5$, ... commute. Hence, they have either an axis in common or a center in common. The same holds for the involutions of even index, $\rho_2$, $\rho_4$, $\rho_6$, ... Using the duality, we can assume that the involutions of odd index have the same axis. We call it the ``odd axis". Then, there are four possibilities. 

1) the involutions of even index have the odd axis as their axis. Then all the involutions commute and we have the triangular subgroup $C_2^k$.

2) the involutions of even index have a common axis different from the odd axis. We call this axis the even axis. The two axes intersect at a point $p$.  We know that $\rho_i$ and $\rho_{i+3}$ commute. However, they have different axes. Therefore, they need to have the same center. As $p$ is the only point contained in both axis, $p$ is necessarily the center of $\rho_i$ and $\rho_{i+3}$. From rank $6$ and higher, this implies that all of the involutions have the same center $p$. Therefore, they all commute. For the rank $5$, all of the involutions except $\rho_3$ have the same center and we get a triangular subgroup. For the rank $4$, the point $p$ is the center of $\rho_1$ and $\rho_4$. It cannot be the center of $\rho_2$ or $\rho_3$ for  otherwise, we would get a triangular subgroup. Therefore, we are in the following situation where the numbers are the indices of the involutions.
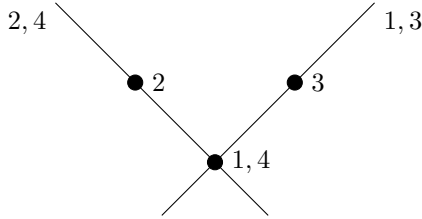
\begin{center}\begin{tikzpicture}
    \draw (45:-1) -- (45:3) node[anchor=north west] {$1,3$} ;
    \draw (135:-1) -- (135:3) node[anchor=north east] {$2,4$} ;
    \draw[fill] (0,0) circle (0.1) ++(0:0.1) node[anchor=west]{$1,4$} ;
    \draw[fill] (45:1.5) circle (0.1) ++(0:0.1) node[anchor=west]{$3$} ;
    \draw[fill] (135:1.5) circle (0.1) ++(0:0.1) node[anchor=west]{$2$} ;
\end{tikzpicture}\end{center}
Due to our previous analysis in rank $3$, we know that $G_1$ and $G_4$ are isomorphic to $E_{q'}^2 : D_{2k}$. For the intersection property to be verified, one has great restrictions on the involutions. Let us make this clearer. If we let
\[ \def\arraystretch{1}
\rho_1 = \begin{pmatrix} 1 & 0 & 0 \\ 0 & 1 & 1 \\ 0 & 0 & 1 \end{pmatrix},
\rho_2 = \begin{pmatrix} 1 & 0 & 0 \\ 0 & 1 & 0 \\ a & 0 & 1 \end{pmatrix},
\rho_3 = \begin{pmatrix} 1 & 0 & b \\ 0 & 1 & 0 \\ 0 & 0 & 1 \end{pmatrix},
\rho_4 = \begin{pmatrix} 1 & 0 & 0 \\ 1 & 1 & 0 \\ 0 & 0 & 1 \end{pmatrix} \]
Then, $G_1$ contains matrices $M(x,ay)$ and $G_4$ contains matrices $M(bx,y)$ for $x,y \in \FF_{q'} = \FF_2[ab]$. Therefore, the intersection of $G_1$ and $G_4$ is $G_{14} \cong D_{2k}$ only when $a \notin \FF_2[ab]$ or equivalently $b \notin \FF_2[ab]$. In this case, we are able to form all the matrices $M(x+bx',y+ay')$ with $x,x',y,y' \in \FF_{q'}$. This yields a normal subgroup $E_{q'}^4$.

3) the involutions of even index have the same center which belong to the odd axis. In this case, this center and the odd axis are preserved by the involutions, a contradiction with our assumption that no point is fixed.

4) the involutions of even index have the same center which does not belong to the odd axis. We call this center the even center. Then $\rho_1$ and $\rho_4$ are supposed to commute, but this is impossible as they have different axes and centers. Indeed, the axis of $\rho_1$ contains the even center but the axis of $\rho_4$ does not. Similarly, the center of $\rho_1$ is contained in the odd axis, but the center of $\rho_4$ is not. Therefore, $d<4$, a contradiction.
\end{proof}

\subsection{Summary}

Tables~\ref{summaryeven} and~\ref{summaryodd} summarize the regular non-degenerate string C-groups $(G,S)$ with $G\leq \PSL(3,q)$ that we found in the previous subsections. 
Note that $\Sym(4)$ appears in the case $p=2$ in the form $E_2^2 : D_6$. Similarly, when $n$ is even, $\Alt(5)$ appears as $\PSL(2,4)$. Among those subgroups, some do non give directly regular polytopes. For example, $\Alt(5)$ and $\PSL(2,q')$ are not directly regular. 
These cannot be considered when trying to construct chiral polytopes of rank six by Lemma~\ref{lem:2.3}.
In Table~\ref{summaryodd}, $He_p$ is the Heisenberg group of order $p^3$ and $q'$ is the smallest root of $q$ such that $k|q' \pm 1$. 

\begin{table}
\begin{tabular}{|c|c|}
\hline
  Schl\"afli type  & $G$ \\
    \hline
    \hline
   $[4]$ & $D_8$ \\
   $[k]$ with $k|q-1$ & dihedral subgroups of $D_{2(q-1)}$ \\
   $[k]$ with $k|q+1$ & dihedral subgroups of $D_{2(q+1)}$ \\
    \hline
    $[k,4]$ and $[k,k]$ & $E_{q'}^2 : D_{2k}$ \\
   All of rank 3 (see~\cite{CPS2008, ls07})  & $\PSL(2,q')$\\
    \hline
    $[4,k,4]$ & $E_{q'}^4 : D_{2k}$ \\
    \hline
\end{tabular}
\caption{The non-degenerate string C-groups $(G,S)$ with $G\leq \PSL(3,q)$ ($q$ even).}\label{summaryeven}
\end{table}

\begin{table}
\begin{tabular}{|c|c|}
\hline
  Schl\"afli type &$G$\\
    \hline
    \hline
   $[k]$ with $k|2p$&
   dihedral subgroups of $D_{4p}$ 
   \\
   $[k]$ with $k|q-1$
&dihedral subgroups of $D_{2(q-1)}$ 
\\
    $[k]$ with $k|q+1$ &
    dihedral subgroups of $D_{2(q+1)}$ 
    \\
    \hline
    $[2p,p]$ &
    $D_{2p}^2$ 
    \\
    $[2p,4]$ and $[4,4]$ &
    $D_{2p}\wr C_2$
    \\
    $[2p,p]$ and $[2p,2p]$ &
    $He_p : C_2^2$
    \\
    $[k,4]$ and $[k,k]$     &
    $E_{q'}^2 : D_{2k}$
    \\
    $[3,3]$ and $[3,4]$&$\Sym(4)$\\
    $[5,3]$ and $[5,5]$&$\Alt(5)$\\
    All  of rank 3 (see~\cite{CPS2008})& $\PSL(2,q') (:2)$  \\
    \hline
    $[p,2p,p]$ &
    $He_p : C_2^2$  
    \\
%    $[p,2,p]$ &
%    $D_{2p}^2$
%    \\
    $[3,3,3]$ & $\PGL(2,5)$\\
    $[3,5,3]$ & $\PSL(2,11)$\\
    $[5,3,5]$ & $\PSL(2,19)$\\ \hline
    \hline
\end{tabular}
\caption{The non-degenerate string C-groups $(G,S)$ with $G\leq \PSL(3,q)$ ($q$ odd).}\label{summaryodd}
\end{table}

\section{Chiral polytope of rank at least $6$}\label{srank6}

Now that we have classified all string C-groups of rank at least four for subgroups of $\PSL(3,q)$, we can use this classification and Lemma~\ref{lem:2.3} to prove Theorem~\ref{rank6}.

We begin by stating a general lemma which will permit to considerably reduce the number of cases to consider.
\begin{lemma}\label{notwo}
    Let $G$ be the group of automorphisms of a chiral polytope such that one of the generators of the chiral polytope has order $2$. Then $G$ has a normal subgroup $H$ of index $2$ and $G\cong H :C_2$.
\end{lemma}
\begin{proof}
    If $G = \left\langle \sigma_1, \sigma_2, \cdots, \sigma_{r-1}\right\rangle$ gives a chiral polytope and there is $i\in\{1, \ldots, r-1\}$ such that $\sigma_i$ has order $2$ then $\left\langle \sigma_1, \sigma_2, \cdots, \sigma_{i-1},\sigma_{i+1}, \cdots, \sigma_{r-1}\right\rangle$ and $\left\langle\sigma_i\right\rangle$ gives the two subgroups which decompose $G$. Indeed, since $\sigma_i^2 = \epsilon$, $(\sigma_{i-1}\sigma_i)^2 = \epsilon$ and $(\sigma_i\sigma_{i+1})^2 = \epsilon$, we see the conjugation by $\sigma_i$ sends $\sigma_{i-1}$ and $\sigma_{i+1}$ to their inverse. Moreover, $\sigma_i$ commutes with the involutions $\sigma_m\sigma_{m+1}\cdots\sigma_{i-1}$ for $m < i -1$ and with the involutions $\sigma_{i+1} \cdots \sigma_{n-1}\sigma_n$ for $n > i +1$.
\end{proof}
Thanks to this lemma, we can now easily prove the following theorem.
\begin{theorem}\label{norank7}
The maximal rank of a chiral polytope with automorphism group $\PSL(3,q)$ is at most six.
\end{theorem}
\begin{proof}
We know that if $\left\langle \sigma_1, \cdots, \sigma_{r-1} \right\rangle$ gives a chiral polytope of type $[k_1,k_2,\cdots,k_{r-1}]$, then $\left\langle \tau_{1,2}, \tau_{1,3}, \cdots, \tau_{1,r-1} \right\rangle$ gives a directly regular polytope of type $[k_3,k_4,\cdots,k_{r-1}]$ and similarly for the dual. Therefore, any chiral polytope of rank $r$ can be constructed  from a directly regular polytope of rank $r-2$. To construct a chiral polytope of rank at least $7$ for $\PSL(3,q)$, one would need a directly regular polytope of rank at least 5 for a subgroup of $\PSL(3,q)$. The only such polytopes appear when $q$ is even, but they have a generator of order $2$ which would contradict Lemma~\ref{notwo}.
\end{proof}
Before proving Theorem~\ref{rank6}, we recall the following result due to Schulte and Weiss.
\begin{lemma}\cite[Lemma 5]{SchulteWeiss}\label{SW}
Let $G = \left\langle \sigma_1, \sigma_2, \cdots, \sigma_{r-1}\right\rangle$ be the automorphism group of a chiral polytope. Then $\sigma_i\sigma_j = \sigma_j\sigma_i$ whenever $|i-j|>2$.
\end{lemma}
\begin{proof}[Proof of Theorem~\ref{rank6}]
Suppose $G\langle \sigma_1, \ldots, \sigma_{r-1}\rangle \leq \PSL(3,q)$ is the automorphism group of a chiral polytope $\mathcal{P}$ of rank $r$.
By Theorem~\ref{norank7}, we know that $r\leq 6$.
Suppose then that $r = 6$.
For a chiral polytope of rank $6$, one would need a directly regular polytope of rank $4$ without generator of order $2$ as $4$-facet. By Tables~\ref{summaryeven} and~\ref{summaryodd}, the only possible Schl\"afli types for $\mathcal{P}$ are therefore as follows.
\begin{itemize}
    \item for $q$ even, $[4,k,4,k',4]$
    \item for $q$ odd, $[p,2p,p,2p,p]$
    \item for $p = 5$, $[3,3,3,3,3]$
\end{itemize}
The last case gives the alternating group. Therefore, we only have to consider the first two cases.

Suppose $q$ even. Then, as it is shown in Section~\ref{qeven}, we can suppose that
\[ \def\arraystretch{1}
\tau_{1,2} = \begin{pmatrix} 1 & 0 & 0 \\ 0 & 1 & 1 \\ 0 & 0 & 1 \end{pmatrix},
\tau_{1,3} = \begin{pmatrix} 1 & 0 & 0 \\ 0 & 1 & 0 \\ a & 0 & 1 \end{pmatrix},
\tau_{1,4} = \begin{pmatrix} 1 & 0 & b \\ 0 & 1 & 0 \\ 0 & 0 & 1 \end{pmatrix},
\tau_{1,5} = \begin{pmatrix} 1 & 0 & 0 \\ 1 & 1 & 0 \\ 0 & 0 & 1 \end{pmatrix}. \]
Note that the point $(0,1,0)$ belongs to the four axes and that if we take the axes of two consecutive generators, it is the only point they have in common. But then
\[ \def\arraystretch{1}
\tau_{4,5} = \begin{pmatrix} 1 & 0 & 0 \\ 1 & 1 & 0 \\ a & 0 & 1 \end{pmatrix} \text{ and }
\tau_{3,5} = \begin{pmatrix} 1 & 0 & 0 \\ 1 & 1 & 1 \\ 0 & 0 & 1 \end{pmatrix}\]
have two axes with only the point $(0,1,0)$ in common. It means that the other generators of the regular polytope preserve the point $(0,1,0)$. Therefore, our generated subgroup preserves this point and cannot be the group $\PSL(3,q)$ nor $\PSU(3,\sqrt{q})$ (assuming $q$ is a square).

Suppose $q$ odd. By the results obtained in Section~\ref{qodd}, we can assume
\[ \def\arraystretch{1}
\tau_{1,2} = \begin{pmatrix} 1 & 0 & 0 \\ 0 & -1 & 0 \\ 0 & 0 & -1 \end{pmatrix},
\tau_{1,3} = \begin{pmatrix} 1 & 0 & 0 \\ 1 & -1 & 0 \\ 0 & 0 & -1 \end{pmatrix},\]
\[\tau_{1,4} = \begin{pmatrix} -1 & 0 & 0 \\ 0 & -1 & 0 \\ 0 & 1 & 1 \end{pmatrix},
\tau_{1,5} = \begin{pmatrix} -1 & 0 & 0 \\ 0 & -1 & 0 \\ 0 & 0 & 1 \end{pmatrix} \]
We get
\[ \def\arraystretch{1}
\tau_{4,5} = \begin{pmatrix} -1 & 0 & 0 \\ -1 & 1 & 0 \\ 0 & 0 & -1 \end{pmatrix},
\tau_{3,5} = \begin{pmatrix} -1 & 0 & 0 \\ 0 & 1 & 0 \\ 0 & 0 & -1 \end{pmatrix}\]
Then, $\tau_{2,5}$ has the same axis as $\tau_{1,5}$ and commute with $\tau_{4,5}$. This gives
\[ \def\arraystretch{1}
\tau_{2,5} = \begin{pmatrix} -1 & 0 & 2a \\ 0 & -1 & a \\ 0 & 0 & 1 \end{pmatrix}\]
This yields
\[ \def\arraystretch{1}
\sigma_1 = \begin{pmatrix} 1 & 0 & -2a \\ 0 & 1 & -a \\ 0 & 0 & 1 \end{pmatrix},
\text{ and }
\sigma_5 = \begin{pmatrix} 1 & 0 & 0 \\ 0 & 1 & 0 \\ 0 & -1 & 1 \end{pmatrix} \]
As $\sigma_1$ and $\sigma_5$ do not commute, they contradict Lemma~\ref{SW}. Therefore, there is no chiral polytope of rank $6$ for $\PSL(3,q)$ or any of its subgroups. 
\end{proof}
\section{Conclusion and open problems}\label{conclusion}
As we explained in the introduction, it is known which finite simple groups are automorphism groups of abstract regular polytopes~\cite{Adrien}.
Our result is a contribution towards getting a similar result for chiral polytopes.

The following groups are known to be automorphism groups of at least one chiral polytope:
all finite simple groups except $\Alt(7)$, $\PSL(3,q)$ with $q\leq 4$, $\PSU(3,q)$ with $q\neq 11$, $\PSL(2,q)$ with $q\equiv 3 \mod 4$ and $q$ not satisfying any of the conditions of~\cite[Theorem 1, case (4)]{LM2017}.

This paper's contribution to the above is for the groups $\PSL(3,q)$.
Indeed, our Theorem~\ref{rank4} shows that for every $q\geq 5$, the group $\PSL(3,q)$ is the automorphism group of at least one chiral polytope of rank four.

The following groups are known not to be automorphism groups of at least one chiral polytope:
$\Alt(7)$, $\PSL(2,q)$ with $q\in \{4, 5, 7, 9, 11, 23, 27, 43, 47, 67, 81, 83
\}$ (see~\cite[Table 2]{LM2017}), $\PSL(3,q)$ with $q\leq 9$~\cite{BLT2019}.

The following groups are not known to be or not to be automorphism groups of at least one chiral polytope:
$\PSU(3,q)$ with $q\geq 11$, $\PSL(2,q)$ with $q=p^d \equiv 3 \mod 4$ and $d>1$ not a prime power.

Based on the data collected in~\cite{BLT2019}, there are chiral polytopes of rank five for $\PSU(3,11)$. All groups $\PSU(3,q)$ with $q < 11$ as well as $q=13$ or $16$ have no chiral polytopes.
Our Theorem~\ref{rank6} combined with the results in~\cite{BC2019} show that if a group $PSU(3,q)$ is the automorphism group of a chiral polytope, this polytope has either rank four or five.

For the open cases for $\PSL(2,q)$, we refer to the Conjecture made by Leemans, Moerenhout and O'Reilly-Regueiro in~\cite{LM2017}.
\bibliographystyle{amsplain}

\end{document}